\documentclass[11pt]{article}
\usepackage{bbm}
\usepackage{amsfonts}
\usepackage{pifont}
\usepackage{bbding}
\usepackage{latexsym}
\usepackage{mathrsfs}
\usepackage{amsmath}
\usepackage{cite}
\usepackage{amsthm}
\usepackage{amssymb}
\usepackage{graphicx}

\newtheorem{theorem}{\indent Theorem}[section]
\newtheorem{corollary}{\indent Corollary}[section]
\newtheorem{proposition}{\indent Proposition}[section]
\newtheorem{definition}{\indent Definition}[section]
\newtheorem{lemma}{\indent Lemma}[section]
\newtheorem{remark}{\indent Remark}[section]
\numberwithin{equation}{section}

\date{}
\begin{document}
\begin{center}
{\Large \bf Global random attractor for 3D stochastic Navier-Stokes equation with nonlinear colored noise$^{\small\mbox{\ding{73}}}$}\\
\vspace{0.5cm} {Xingjie Yan}\\\vspace{6pt}
{\small {\em{ Department of Mathematics, China University of Mining
and Technology,\\ Xuzhou, Jiangsu, 221116, People's Republic of
China}}}

\vspace{0.5cm} {Rong Yang$^{*}$}\\\vspace{6pt}
{\small {\em{ School of Mathematics, Statistics and Mechanics, Beijing University of Technology, Ping Le Yuan 100,\\ Chaoyang District, Beijing, 100124, People's Republic of
China}}}

\vspace{0.5cm} {Alain Miranville}\\\vspace{6pt}
{\small {\em{ Laboratoire de Math\'{e}matiques Appliqu\'{e}es du Havre, Universit\'{e} Le Havre Normandie,\\ 76600 Le Havre, France }}}

\end{center}
\footnote[0] {\hspace*{-0.58cm}$^{\tiny\mbox{\ding{73}}}$ This work
is partially supported by the National Nature Science Foundation of China grant (11501560), and partially supported by the
Fundamental Research Funds for the Central Universities, China under Grant No. 2024KYJD2001.\\
$^*$ Corresponding author. \\
\textit{E-mail address}: yanxj04@163.com (X. Yan), ysihan2010@163.com (R. Yang), alain.miranville@univ-lehavre.fr (A. Miranville).}

\begin{abstract}
In this paper, we first introduce the definitions of random evolutionary system that associate with random evolutionary semigroup and the corresponding global weak or strong random attractor. Then we establish the existence result about global weak or strong random attractor and it's properties like invariance and weak or strong tracking features. Finally, we use our established results to the 3D stochastic Navier-Stokes equation with nonlinear colored noise.


\end{abstract}

\section{Introduction}

Stochastic partial differential equations are powerful tool for understanding
and investigating mathematically hydrodynamic and turbulence theory. It is not just add a noise term on the deterministic partial differential equations. Indeed, many rigorous information on questions
of turbulence theory like memory effect and uncertainty might be obtained from these stochastic versions \cite{8,9}. It should be emphasized that the random effects will exhibit different behaviors in the equations.

Traditionally, if the whole dynamical system is well-posedness, then  semigroup  and  global attractor can be defined to analyze the dynamical behavior. However, more and more examples indicate that ill-posedness problems will become increasingly common. For example,  the uniqueness of weak solution to the 3D Navier-Stokes equations, 2D Euler equations and 3D magnetohydrodynamic equations is unknown. To analyze the dynamical behavior of the above ill-posedness problems, usually there are the following methods used: the first one is the multi-value semigroup and global attractor \cite{12}, the second one is the generalized semigroup and global attractor \cite{13}, the third one is the trajectory attractor \cite{14} and the last one is the evolutionary system and global attractor \cite{1}. The difference between them is how they define semigroup, use all trajectories of system by initial points and end points or single trajectory or define translation semigroup on space that consist of all trajectories of system. 

Ball's generalized semigroup is very similar to multi-value semigroup, except for concatenate and upper semicontinuous properties of trajectory when applied to  3D Navier-Stokes equations, but these properties are open for Leray-Hopf weak solutions to the 3D Navier-Stokes equation. The phase space of trjectory attractor is not the original phase space  that  the connection with
the physical meaning. More importantly, how to extend the deterministic results of ill-posedness problems to stochastic case is a topic of considerable interest. We refer to \cite{11, CWWZ2023, GW2018, GW2020, 33, WW2015, WY2019} and the references therein for the study of multivalued random dynamical system. If we study long time behavior of  3D stochastic Navier--Stokes equations, the theory of above multivalued random dynamical system isn't suitable completely, the reason is very similar to deterministic case, thus it is very necessary for us extend theory of evolutionary system and global attractor \cite{1} to random case.

The study of dynamical behavior of 3D stochastic Navier--Stokes equations is an interesting and challenging problem. Especially, the existence of global attractor is still an open problem though some partial results have been obtained.  The first results presented in \cite{3,4} concerned the existence of a trajectory attractor for the Navier--
Stokes system perturbed by multiplicative and additive noises, respectively. A similar
approach is used in \cite{5}. In \cite{6}, the conditional result of random attractor was obtained to the 3D stochastic Navier--Stokes system perturbed by an additive noise. The authors in \cite{7} obtained a weak global attractor for 3D stochastic Navier-Stokes equations perturbed by a multiplicative Brownian noise in help with a modified equation. The reason for they choose a 3D stochastic globally modified Navier--Stokes equations is that the weak solution of this equation satisfy energy inequality. In summary, as to the global random attractor for the 3D stochastic Navier-Stokes equations, the weak global random attractor was obtained with no additional condition. However,  in order to obtain the strong global random attractor,  all the  results are obtained under certain conditions, even for the deterministic case.

In this paper, the  main intention is to  study the existence of global random attractor and tracking properties for the 3D stochastic Navier-Stokes equation with nonlinear colored noise.  The main features and difficulties can be proposed as follows.
\begin{enumerate}
	\item[(I)] Motivated by the deterministic case in \cite{1,2},  we extend the theory of deterministic evolutionary system  to the random case.    

	\item[(II)] We use the established abstract results to study the dynamical behavior of Leray-Hopf weak solution of random 3D Navier-Stokes equation, the noise term is nonlinear colored instead of white.
	
\end{enumerate}

The organization of this paper is as follows. In Section 2, we introduce the definition of random evolutionary system that contains all information of stochastic flow, then we also define the random evolutionary semigroup. In Section 3, naturally, the definition of weak and strong random global attractor for random evolutionary semigroup is presented. Then we give theorems of existence of weak random global attractor. In Section 4, with the help of strong random asymptotic compactness, the weak random global attractor obtained in section 3 becomes strong one. In Section 5, weak and strong invariance and tracking features are obtain for weak and strong random global attractor respectively. In Section 6, we apply the abstract results established above to the 3D Navier-Stokes equation with nonlinear colored noise and analyze it's dynamical behavior.

\section{Random evolutionary system}
In this section, some notations of evolutionary system are stated, which agree with  \cite{1,2}. Also for the basic theory of random dynamical systems and the existence of random attractors, we refer the reader to \cite{8,10,11}. The combining of the theories of above two dynamical systems leads to the definition of random evolutionary system. Then we give some properties of random evolutionary semigroup.

Let $X$ be a metric space with strong metric $d_s$, it is abbreviated as $(X,d_s(\cdot,\cdot))$ (sometimes $X$  is a Banach space with norm $\|\cdot\|_X$, there is no misunderstand according to the above and following content). We say that $d_w$ is a weak metric on $X$ if it satisfies the following conditions:
\begin{enumerate}
\item[(a)]$X$ is $d_w-$compact;
\item[(b)] $d_s(u_n,v_n)\rightarrow0$ as $n\rightarrow\infty$ for some $u_n,v_n\in X$ implies that $d_w(u_n,v_n)\rightarrow0$ as $n\rightarrow\infty$.
\end{enumerate}

The black dote $\bullet$ represents strong or weak, abbreviated as $\bullet=s,w$. So $\bar{A}^\bullet$ denotes the closure of a set $A\subset X$ in the topology generated by $d_\bullet$. By functional analysis, any strongly compact set is weakly compact, and any weakly closed set is strongly closed.

Space $C([a,b];X_\bullet)$ is the collection  of $d_\bullet-$continuous $X-$valued functions on interval $[a,b]$ endowed with the metric
$$d_{C([a,b];X_\bullet)}(u,v)=\max_{t\in[a,b]}d_\bullet(u(t),v(t)),\quad \forall u,v\in C([a,b];X_\bullet).$$

Let $C([a,\infty);X_\bullet)$ be the space of $d_\bullet-$continuous $X-$valued functions on interval $[a,\infty)$ endowed with the metric
$$d_{C([a,\infty);X_\bullet)}(u,v)=\sum_{T\in \mathbb{N}}\frac{1}{2^{T}}\frac{\sup\{d_\bullet(u(t),v(t)):a\leq t\leq a+T\}}{1+\sup\{d_\bullet(u(t),v(t)):a\leq t\leq a+T\}}.$$

The convergence in $C([a,\infty);X_\bullet)$ means  locally uniform convergence, i.e., for any $u_n \rightarrow u$ in $C([a,\infty);X_\bullet)$  as $n\rightarrow\infty$ implies that
$u_n|_{[T_1,T_2]}\rightarrow u|_{[T_1,T_2]}$ in $C([T_1,T_2];X_\bullet)$, for any $T_1,T_2\in [a,\infty)$.

Let $X$ and $Y$ be two metric spaces with Borel $\sigma-$algebras $\mathcal{B}(X)$ and $\mathcal{B}(Y)$, respectively.  Let $P(X)$ be the set of all subsets of $X$. Recall that a mapping $G:X\rightarrow P(Y)$ is measurable with respect to $\mathcal{B}(X)$ if the inverse image of any open subset of $Y$ under $G$ is a Borel subset of $X$.
\begin{definition}
Suppose  $X$ is a metric space and $Y$ is a Banach spaces. A multivalued function $G:X\rightarrow 2^{Y}$ is said to be weakly upper semicontinuous at $x_0\in X$ if for any $x_n\rightarrow x_0$ in $X$ and $z_n\in G(x_n)$, there exist a subsequence $z_{n_k}$ of $z_n$ such that $z_{n_k}\rightarrow z_0$ weakly in $Y$, here $z_0\in G(x_0)$. In addition, if $G$ is weakly upper semicontinuous at every $x\in X,$ we say that $G$ is weakly upper semicontinuous.
\end{definition}

\begin{theorem}(\cite[Theorem 2.3]{33})\label{2.1} Let $X$ be a metric space and $Y$ a separable Banach space. If $G:X\rightarrow 2^{Y}$ is weakly upper semicontinuous, then $G$ is measurable with respect to $\mathcal{B}(X)$.
\end{theorem}
Define the set
$$\mathcal{T}:=\{I:I=[T,\infty)\subset \mathbb{R},~\mathrm{or}~ I=(-\infty,+\infty)\},$$
and for each $I\subset \mathcal{T}$, let $\Xi(I)$ denote the set of all $X-$valued functions on $I$.
Assume that $(\Omega, \mathcal{F}, P, \{\theta_t\}_{t\in \mathbb{R}})$ is a metric dynamical system, where $(\Omega, \mathcal{F}, P)$ is a probability space and $\theta:\mathbb{R}\times\Omega\rightarrow\Omega$ is a measure-preserving group of translations on $\Omega$. Following, we give the difinition of random evolutionary system that contains all the information of stochastic flow of the system we will consider, we can define different random evolutionary system according to different system. 
\begin{definition}\label{evsys}
A family of $\mathcal{B}(\mathbb{R}^+)\times\mathcal{F}\times\mathcal{B}(X)$ measurable maps $\mathcal{E}^{\theta}_{\omega}$, which satisfies that subset $\mathcal{E}^{\theta}_{\omega}(I)\subset\Xi(I)$ for every $\omega\in\Omega$ and $I\in\mathcal{T}$,  will be called a random evolutionary system if the following conditions hold:
\begin{enumerate}
\item[1.]$\mathcal{E}^{\theta}_{\omega}([T,\infty))\neq\emptyset,$ $\forall T\in \mathbb{R}, \omega\in\Omega,$ $u\in \mathcal{E}^{\theta}_{\omega}([T,\infty))$ is $\mathcal{B}(\mathbb{R}^+)\times\mathcal{F}\times\mathcal{B}(X)$ measurable;
\item[2.]$\mathcal{E}^{\theta}_{\omega}(I+s)=\{u(\cdot):u(\cdot-s)\in \mathcal{E}^{\theta}_{\theta_s\omega}(I)\},\forall s\geq0$;
\item[3.]$\{u(\cdot)|_{I_2}:u(\cdot)\in \mathcal{E}^{\theta}_{\omega}(I_1)\}\subset\mathcal{E}^{\theta}_{\omega}(I_2)$ for all $I_1,I_2\in\mathcal{T}$ and $I_2\subset I_1, \omega\in\Omega$;
\item[4.]$\mathcal{E}^{\theta}_{\omega}((-\infty,\infty))=\{u(\cdot):u(\cdot)|_{[T,\infty)}\in \mathcal{E}^{\theta}_{\omega}([T,\infty)),\forall  T\in \mathbb{R}\}.$
\end{enumerate}
\end{definition}

We will denote $\mathcal{E}^{\theta}_{\omega}(I)$ as the set of all random trajectories with respect to sample $\omega$ on the time interval $I$. Random trajectories in $\mathcal{E}^{\theta}_{\omega}((-\infty,+\infty))$ will be called complete. For every $t\geq0$, $  \omega\in\Omega$, we define a map
$$ R(t, \omega):P(X)\rightarrow P(X),$$
$$ R(t, \omega)A:=\{u(t):u(0)\in A, u\in\mathcal{E}^{\theta}_{\omega}([0,\infty))\}, \quad\forall A\in P(X).$$
Note that the assumptions on $\mathcal{E}^{\theta}_{\omega}$ imply that $ R(t, \omega)$ enjoys the following property:
\begin{enumerate}
\item[(1)]$ R(0, \omega)\cdot$ is an identity map on $X$.
\item[(2)]$R(\cdot, \cdot)\cdot$ is $\mathcal{B}(\mathbb{R}^+)\times\mathcal{F}\times\mathcal{B}(X)$ measurable;
\item[(3)]$ R(t, \omega)A(\omega)\subset R(t, \theta_s\omega)R(s, \omega)A(\omega), \forall A\in P(X), t\geq s\geq0,  \omega\in\Omega,$
\end{enumerate}
here we say that $\{ R(t, \omega)\}=\{ R(t,\omega): t\in \mathbb{R}^+,  \omega\in\Omega\}$ is a random evolutionary semigroup.

 As a consequence of Theorem \ref{2.1}, the above map enjoys the following property.
\begin{lemma}\label{measurable}
Let $\{ R(t, \omega)\}=\{ R(t, \omega): t\in \mathbb{R}^+,  \omega\in\Omega\}$ is a random evolutionary semigroup. Suppose that $\Omega$ is a metric space and $X$ a separable Banach space. If the mapping $R(\cdot, \cdot)\cdot$ is weakly upper semicontinuous, then $R(\cdot, \cdot)\cdot$ is $\mathcal{B}(\mathbb{R}^+)\times\mathcal{F}\times\mathcal{B}(X)$ measurable.
\end{lemma}



\section{Global random attractor}
In this section, we present some abstract results of global random attractor of random evolutionary system. Considering an arbitrary random evolutionary system $\mathcal{E}^{\theta}_{\omega}$, for a set $A\subset X$ and $r>0$, denote $$B_\bullet(A,r)=\{u\in X:d_\bullet(u,A)<r\},$$ where
$$d_\bullet(u,A):=\inf_{x\in A}d_\bullet(u,x).$$

Let $(\Omega, \mathcal{F}, P, \{\theta_t\}_{t\in \mathbb{R}})$ be a metric dynamical system and $\{ R(t, \omega)\}$ be a random evolutionary semigroup and a family of random set $\{K(\omega)\}\in P(X)$,  $\omega\in\Omega$ means that every $K(\omega)$ belongs to $ P(X)$.
\begin{definition}
A family of set $\{K(\omega)\}$ is called random absorbing set, if for any given $D\in  P(X)$, there exists $t_0=t_0(D)$, such that
$$R(t,\theta_{-t}\omega)D\subset K(\omega),~\mathrm{when} ~t\geq t_0.$$
\end{definition}

\begin{definition}
 A family of sets $\{K(\omega)\}\subset X$ we call random  $d_\bullet-$attracting set is that it random  $d_\bullet-$attract $D\in  P(X)$ under $R(t,\omega)$ if
$$\lim_{t\rightarrow\infty}d_\bullet(R(t,\theta_{-t}\omega)D,K(\omega))=0,\omega\in\Omega.$$
\end{definition}
\begin{definition}
A family of set $\{\mathcal{A}_\bullet(\omega)\}$ is called global random attractor for a random evolutionary semigroup $\{R(t,\omega)\}$ if it is compact and satisfying that
\begin{enumerate}
\item[(1)]measurable, $\mathcal{A}_\bullet:\Omega\rightarrow P(X)$ is measurable with respect to $\mathcal{F}$;
\item[(2)]invariance, $R(t,\omega)\mathcal{A}_\bullet(\omega)=\mathcal{A}_\bullet(\theta_t\omega), \forall \omega\in\Omega$;
\item[(3)]random $d_\bullet-$attracting set.
\end{enumerate}
\end{definition}

\begin{remark}
  $\{K(\omega)\}$ is compact means that  $K(\omega)$ is a compact set for every $w\in\Omega$.
\end{remark}

\begin{remark}
In the above definition, if $\{\mathcal{A}_\bullet(\omega)\}$ exist, then it is minimal. The minimal means that if there exists another a family of $d_\bullet-$closed and random $d_\bullet-$attracting set, then it  contains $\{\mathcal{A}_\bullet(\omega)\}$.
\end{remark}


\begin{definition}
The  $\Omega_\bullet-$limit of a set $A\in  P(X)$ is defined by
$$\Omega_\bullet(A,\omega):=\bigcap_{s\leq 0}\overline{\bigcup_{t\leq s}R(t,\theta_{-t}\omega)A}^\bullet.$$
Or equivalently
\begin{eqnarray*}
\Omega_\bullet(A,\omega)=\{y\in X:~\mathrm{ there~ are~ sequences~} s_k\geq 0,s_k\rightarrow\infty~ \mathrm{as}~ k\rightarrow\infty ~\mathrm{and} ~x_k\in R(s_k,\theta_{-s_k}\omega)A,\\ \mathrm{~such~ that~}
x_k\rightarrow y~ \mathrm{in}~ d_\bullet \mathrm{~metric~ as~ }k\rightarrow\infty\}. \quad\quad\quad \quad\quad\quad \quad\quad\quad \quad\quad\quad \quad\quad\quad \quad
\end{eqnarray*}
\end{definition}

Now, we are ready to present some properties of $\Omega_\bullet-$limit set that immediately follow from the definition.

\begin{lemma}\label{l3.1}
Let $A\in  P(X)$ and $ \omega\in\Omega$, then we have
\begin{enumerate}
\item[(1)]$\Omega_\bullet(A,\omega)$ is $d_\bullet-$closed;
\item[(2)]$\Omega_s(A,\omega)\subset \Omega_w(A,\omega)$;
\item[(3)]If $\Omega_w(A,\omega)$ is strongly compact and strongly attracts $A$, then $\Omega_s(A,\omega)=\Omega_w(A,\omega)$.
\end{enumerate}
\end{lemma}

\begin{proof}
Part $(1)$ is trivial by definition of $\Omega_\bullet-$limit set. Part $(2)$ follows from the relationship between strong and weak convergence. So it is only left to prove part $(3)$.

Taking any $x\in \Omega_w(A,\omega)$, by the definition of weak $\Omega-$limit set, there exists a sequence $s_k\geq t$, $s_k\rightarrow\infty$ as $k\rightarrow\infty$ and a sequence $x_k\in R(s_k,\theta_{-s_k}\omega)A$, such that $x_k\rightarrow x$ weakly as $k\rightarrow\infty$. Since  $\Omega_w(A,\omega)$ strongly attracts $A$, then there exists a sequence $a_k\in \Omega_w(A,\omega)$ such that
$d_s(x_k,a_k)\rightarrow0$ as $k\rightarrow\infty.$ Thus $a_k\rightarrow x$ weakly as $k\rightarrow\infty.$ Because of $\Omega_w(A,\omega)$ is strongly compact, so this convergence is also strong. Therefore we have $x_k\rightarrow x$ strongly as $k\rightarrow\infty$ and $x\in \Omega_s(A,\omega)$, which completes the proof.
\end{proof}

Usually $\Omega_\bullet(A,\omega)$ is the minimal $d_\bullet-$closed and random $d_\bullet-$ attracting set. Next, we consider the $\mathcal{F}$ measurability of $\Omega_\bullet(A,\omega)$, which is a result in \cite{33}.
\begin{lemma}
If $K\in P(X)$ is the random absorbing set of random evolutionary semigroup $\{ R(t, \omega)\}$ and the mapping $\omega\rightarrow R(t,\omega)K(\omega)$ is measurable with respect to $\mathcal{F}$ for every $t\in \mathbb{R}^+$, then the $\Omega_\bullet-$limit set $\Omega_\bullet(K,\omega)$ is also measurable with respect to $\mathcal{F}$.
\end{lemma}

\begin{lemma}\label{l3.2}
Let $\{A_\bullet(\omega)\}$ be a  family of $d_\bullet-$closed, random $d_\bullet-$attracting set and measurable with respect to $\mathcal{F}$. Then we have
$$\Omega_\bullet(A,\omega)  \subset A_\bullet(\omega),\quad \forall A\in  P(X), \omega\in\Omega.$$
\end{lemma}

\begin{proof}
On the contrary, we assume that there exists a point $x\in \Omega_\bullet(A,\omega)  \setminus A_\bullet(\omega)$. By the definition of $\Omega_\bullet-$limit set, there exists  a sequence $s_k\geq 0$, $s_k\rightarrow\infty$ as $k\rightarrow\infty$ and $x_k\in A$ such that
$$\lim_{k\rightarrow\infty}d_\bullet(R(s_k,\theta_{-s_k}\omega)x_k,x)=0.$$

But $\{A_\bullet(\omega)\}$ is random $d_\bullet-$attracting set, so we have
$$\lim_{k\rightarrow\infty}d_\bullet(R(s_k,\theta_{-s_k}\omega)x_k,A_\bullet(\omega))=0,$$
which contradicts with the choice of $x$. The proof is complete.
\end{proof}


\begin{theorem}\label{t3.1}
If $\{\mathcal{A}_\bullet(\omega)\}$ exists, then
$$\mathcal{A}_\bullet(\omega)=\Omega_\bullet(A,\omega),\quad A\subset P(X).$$
\end{theorem}
\begin{proof}
Thanks to Lemma \ref{l3.2}, we have $\Omega_\bullet(A,\omega)\subset \mathcal{A}_\bullet(\omega)$. Since $\{\mathcal{A}_\bullet(\omega)\}$ is the minimal random  $d_\bullet-$attracting set, so the reverse conclusion also holds. We complete the proof.
\end{proof}

We present the characterization of the existence of $d_\bullet-$global random attractor.

\begin{theorem}\label{t3.2}
Suppose that the random evolutionary semigroup $\{R(t,\omega)\}$ is weak upper semicontinuous.
For $\omega\in \Omega$, a family of $d_\bullet-$global random attractor $\{\mathcal{A}_\bullet(\omega)\}$  exists if and only if $\Omega_\bullet(A,\omega)$, $A\subset P(X)$ is a random  $d_\bullet-$attracting set.
\end{theorem}
\begin{proof}
Suppose that the left holds, i.e., $\{\mathcal{A}_\bullet(\omega)\}$ exists, then the above theorem implies that $\mathcal{A}_\bullet(\omega)=\Omega_\bullet(A,\omega)$, so $\Omega_\bullet(A,\omega)$ is a random $d_\bullet-$attracting set, we get the right.

If the right holds, by Lemma \ref{l3.1}, $\Omega_\bullet(A,\omega)$ is also  $d_\bullet-$closed set, thanks to Theorem \ref{2.1}, $\Omega_\bullet(A,\omega)$ is measurable with respect to $\mathcal{F}$, then it is only to verify invariance of $\Omega_\bullet(A,\omega)$.

Given any $y\in\Omega_\bullet(A,\omega)$, by definition, there exist $t_n\rightarrow\infty$ when $n\rightarrow\infty$ and $x_n\in K(\theta_{-t_n}\omega)$ such that
$R(t_n,\theta_{-t_n}\omega)x_n\rightarrow y$ weakly in $Y$ when $n\rightarrow\infty$. Thanks to weak upper semicontinuity,  for $t>0$, we have $R(t+t_n,\theta_{-t-t_n}\theta_{t}\omega)x_n$ weakly converges to $R(t,\omega)y$,
hence $R(t,\omega)y\in\Omega_\bullet(A,\theta_{t}\omega)$, $R(t,\omega)\Omega_\bullet(A,\omega)\subset \Omega_\bullet(A,\theta_{t}\omega)$.
The reverse conclusion holds by the definition of $\{\mathcal{A}_\bullet(\omega)\}$.
\end{proof}

Asymptotical compactness is very important for random evolutionary semigroup to obtain the global random attractor.

\begin{definition}
 If for each $s_k\geq0$, with $s_k\rightarrow\infty$ as $k\rightarrow\infty$, and each bounded sequence $\{x_k\}\subset D, D\in P(X)$, the sequence $\{R(s_k,\theta_{-s_k}\omega)x_k\}$ converges in $d_\bullet-$metric, then we say that $\{ R(t, \omega)\}$ is $\bullet-$asymptotically compact.
\end{definition}

Next, we present the existence of weak $\Omega-$limit set $\Omega_w(A,\omega),\omega\in \Omega, A\in P(X)$ and it's properties.

\begin{theorem}\label{t3.3}
Suppose that there exists a family of random absorbing set and the random evolutionary semigroup $\{R(t,\omega)\}$ is weak upper semicontinuous. Let $\omega \in \Omega$ and $A\in  P(X)$ be such that there exists $u\in \mathcal{E}^{\theta}_{\omega}([T,\infty))$ with $u(t)\in A$, $T\in \mathbb{R}$. Then $\Omega_w(A,\omega)$ is a nonempty weakly compact set, invariance and measurable. In addition, $\Omega_w(A,\omega)$  $d_w-$attracts $A$.
\end{theorem}

\begin{proof}
Note first that there exists a time $t_0$ such that
$$\overline{\cup_{t\geq t_0}R(t,\theta_{-t}\omega)A}^w \mathrm{~is~ bounded}. $$
Now for any sequences $\{x_k\}\in A, A\subset X$ and $s_k\geq t_0$, with $s_k\rightarrow\infty$ as $k\rightarrow\infty$, it follows from the fact that $\{R(t,\omega)\}$ is $d_w-$asymptotically compact (this fact follows from $X$ is weakly compact), then there exists a subsequence of $\{R(s_k,\theta_{-s_k}\omega)x_k\}$ converges to some $y\in X$ weakly. Thus $y\in \Omega_w(A,\omega)$ and $\Omega_w(A,\omega)$ is non-empty. Invariance and measurable comes from Theorems \ref{t3.2} and \ref{2.1} respectively.

Next, we prove that $\Omega_w(A,\omega)$ $d_w-$attracts $A$ by contradiction. Assume that there exist  $\varepsilon>0$, a sequence $s_n\rightarrow\infty$ and a sequence $x_n\in A$, such that for all $n\in \mathbb{N}$
$$dist_w(R(s_n ,\theta_{-s_n}\omega)x_n,\Omega_w(A,\omega))\geq\varepsilon.$$
But we have just shown that there must be a subsequence of $\{R(s_n,\theta_{-s_n}\omega)x_n\}$ that converges to an element of $\Omega_w(A,\omega)$, which is a contradiction.

Finally, we show that $\Omega_w(A,\omega)$ is weakly compact. Given a sequence $\{y_k\}\in \Omega_w(A,\omega)$, there are $\{x_k\}\in B$ and $\{s_k\}\geq\min\{t_0,k\}$, such that
$$dist_w(R(s_k ,\theta_{-s_k}\omega)x_k,y_k)\leq\frac1k.$$
Since $\{R( s_k,\theta_{-s_k}\omega)x_k\}$ has a subsequence that weakly converges to an element of $\Omega_w(A,\omega)$, it follows that $\{y_k\}$ has a subsequence that converges to $y\in \Omega_w(A,\omega)$ and hence $\Omega_w(A,\omega)$ is weakly compact. The proof is complete.
\end{proof}

The following result tells us that if the strong global random attractor exists, then it is also the weak one.

\begin{theorem}
Suppose that there exists a family of random absorbing set and the random evolutionary semigroup $\{R(t,\omega)\}$ is weak upper semicontinuous,
then every random evolutionary system always possess a family of weak global random attractor $\{\mathcal{A}_w(\omega)\}$. If strong global random attractor $\{\mathcal{A}_s(\omega)\}$ exists, then we have $\overline{\mathcal{A}_s(\omega)}^w=\mathcal{A}_w(\omega)$, $\forall \omega\in \Omega$.
\end{theorem}
\begin{proof}
According to Theorem \ref{t3.3}, $\Omega_w(A,\omega),A\subset P(X)$ is a $d_w-$closed and random $d_w-$attracting set, so Theorems \ref{t3.1} and \ref{t3.2} imply that $\{\mathcal{A}_w(\omega)\}$ exists and $\mathcal{A}_w(\omega)=\Omega_w(A,\omega).$

Next, we suppose that $\{\mathcal{A}_s(\omega)\}$ exists. Because $\overline{\mathcal{A}_s(\omega)}^w$ is a $d_s-$random attracting set, it is also a $d_w-$attracting and weakly closed set, by Lemma \ref{l3.2}, we have $\Omega_w(A,t)\subset\overline{\mathcal{A}_s(\omega)}^w$. Theorem \ref{t3.1} tolds us that $\mathcal{A}_s(\omega)=\Omega_s(A,\omega)$, so we get that $\Omega_w(A,\omega)\subset \overline{\Omega_s(A,\omega)}^w$. Since it is obvious that $\overline{\Omega_s(A,\omega)}^w\subset \Omega_w(A,\omega)$, hence due to Lemma \ref{l3.1}, we obtain that $A_w(\omega)=\overline{\mathcal{A}_s(\omega)}^w$, which completes the proof.
\end{proof}

\section{ Strong global random attractor}
In order to characterize the strong  global random attractor, similarly to the above section,  we should know about the properties of random $\Omega_s-$limit set, thus  the $s-$asymptotical compactness of random evolutionary semigroup  plays an important role.

\begin{lemma}\label{l4.1}
Let $\{R(t,\omega)\}$ be $s-$asymptotically compact. Also let $A\in  P(X)$ be such that there exists $u\in \mathcal{E}^{\theta}_{\omega}([T,\infty))$ with $u(T)\in A$, $T\in \mathbb{R}$. Then $\Omega_s(A,\omega)$ is a nonempty strongly compact set that $d_s-$random attracts $A$ and $\Omega_s(A,\omega)=\Omega_w(A,\omega)$.
\end{lemma}
\begin{proof}
First of all, since $\{R(t,\omega)\}$ is $s-$asymptotically compact, then for any sequence $x_n\in A$ and any sequence $s_n\rightarrow\infty$ as $n\rightarrow\infty$, $\{R(s_n,\theta_{-s_n})x_n\}$ has a convergence subsequence, and by definition, the limit of this subsequence must be an element of $\Omega_s(A,\omega)$, which shows that $\Omega_s(A,\omega)$ is nonempty.

Next, we prove that $\Omega_s(A,\omega)$ random $d_s-$attracts $A$ by contradiction. Assume that there exist a $\varepsilon_0>0$, a sequence $s_n\rightarrow\infty$ as $n\rightarrow\infty$ and a sequence $x_n\in A$ such that
$$dist_s(R(s_n,\theta_{-s_n})x_n,\Omega_s(A,\omega))\geq\varepsilon_0, ~\mathrm{for~ all }~n\in \mathbb{N}.$$
But we have shown that there exists a subsequence of $\{R(s_n,\theta_{-s_n})x_n\}$ that converges to an element of $\Omega_s(A,\omega)$, which is a contradiction.

Finally, on one hand we know that $\Omega_s(A,\omega)\subset \Omega_w(A,\omega)$. On the other hand, if $x\in \Omega_w(A,\omega)$, by definition, there exist sequence $s_n\geq 0$, $s_n\rightarrow\infty$ as $n\rightarrow\infty$, and a sequence $x_n\in A$ such that $R(s_n,\theta_{-s_n})x_n$ weakly converges to $x$. But $R(t,\omega)$ is $s-$asymptotically compact, so $R(s_n,\theta_{-s_n})x_n$ also strongly converges to $x$ and $x\in \Omega_s(A,\omega)$, hence $\Omega_s(A,\omega)=\Omega_w(A,\omega), \omega\in\Omega$.

The strongly compactness of $\Omega_s(A,\omega)$ comes from the definition of $\Omega_s(A,\omega)$ and $\{R(t,\omega)\}$ is $s-$asymptotically compact. The proof is complete.
\end{proof}

At the end of this section, we have the existence of strong global random attractor.

\begin{theorem}\label{t4.1}
Assume that the random evolutionary semigroup $\{R(t,\omega)\}$ is weak upper semicontinuous,
 it is $s-$asymptotically compact and there exists a family of random absorbing set. Then $\{\mathcal{A}_s(\omega)\}$ exists and $\{\mathcal{A}_s(\omega)\}=\{\mathcal{A}_w(\omega)\}$.
\end{theorem}
\begin{proof}
As a consequence of Lemma \ref{l4.1}, $\Omega_s(A,\omega)$ is a $d_s-$random attracting set and $\Omega_s(A,\omega)=\Omega_w(A,\omega)=\mathcal{A}_s(\omega)$. Thanks to Theorem \ref{t3.2}, strong and weak global random attractor exist and equal, i.e., $\{\mathcal{A}_w(\omega)\}=\{\mathcal{A}_s(\omega)\}$. The proof is complete.
\end{proof}
\section{Invariance and tracking features}

In order to extend the notion of invariance from a classical semiflow to a random evolutionary semigroup, we need the following mapping:
$$\widetilde{R}(t,\omega)A:=\{u(t):u(0)\in A,u\in\mathcal{E}^{\theta}_{\omega}((-\infty,\infty))\},A\in  P(X),t\geq0, \omega\in\Omega.$$

\begin{definition}\label{inv}
A family of sets $\{B(\omega)\}$ is positive invariant if
$$\widetilde{R}(t,\omega)B(\omega)\subset B(\theta_t\omega),\forall t\geq0, \omega\in\Omega.$$

$\{B(\omega)\}$ is invariant if
$$\widetilde{R}(t,\omega)B(\omega)= B(\theta_t\omega),\forall t\geq0, \omega\in\Omega.$$

$\{B(\omega)\}$ is quasi-invariant if for every $a\in B(\omega)$, there exists a complete trajectory $u\in \mathcal{E}^{\theta}_{\omega}((-\infty,\infty))$ with $u(0)=a$ and $u(s)\in B(\theta_{s}\omega)$ for all $s$.
\end{definition}

Combining with Lemma \ref{l3.1}, it implies that $\forall \omega\in \Omega$, if $\{B(\omega)\}$ is quasi-invariant
$$B(\omega)\subset \Omega_s(A,\omega)\subset \Omega_w(A,\omega).$$


\begin{theorem}\label{t5.1}
Let the random evolutionary semigroup $\{R(t,\omega)\}$ is weak upper semicontinuous in $X$ for every $t\in \mathbb{R}^+, \omega\in\Omega$. Then $\{\Omega_w( A,\omega)\}$ is quasi-invariant for every $A\in  P(X)$.
\end{theorem}
\begin{proof}
We known that $\Omega_w( A,\omega)$ is invariance, that is $R(t,\omega)\Omega_w( A,\omega)=\Omega_w( A,\theta_t\omega)$. So for any $x\in\Omega_w( A,\omega)$, there exist $u(t)\in\mathcal{E}^{\theta}_{\omega}([0,\infty))$, such that $u(0)\in \Omega_w( A,\omega)$ and $u(t)\in \Omega_w( A,\theta_t\omega)$, $t\geq0$.

Next let us extend the random trajectory $u(t)\in\mathcal{E}^{\theta}_{\omega}([0,\infty))$ to be complete. Choosing any $x\in \Omega_w( A,\omega)$, the invariance of $\Omega_w( A,\omega)$ implies that there hold
$$R(-1,\omega)x\in \Omega_w( A,\theta_{-1}\omega),$$
such that
$$R(t,\omega)x\in R(t+1,\theta_{-1}\omega)R(-1,\omega)x,~t\in[-1,0].$$
Then we have $R(t,\omega)x\in\Omega_w( A,\theta_{t}\omega)$ for all $t\geq-1.$

Applying the above procedure several times, then we obtain $R(t,\omega)x\in\Omega_w( A,\theta_{t}\omega)$ for all $t\geq-n,n\in \mathbb{N }$. Letting $n$ go to infinity, we construct a random complete trajectory $u(t)\in\mathcal{E}^{\theta}_{\omega}((-\infty,\infty))$, such that $R(t,\omega)x\in \Omega_w( A,\theta_{t}\omega)$ for all $t\in \mathbb{R}$. The proof is complete.
\end{proof}


\begin{theorem}\label{t5.2}
Let $\{B(\omega)\}\in  P(X)$. Then $\{B(\omega)\}$ is invariant if and only if $\{B(\omega)\}$ is positively invariant and quasi-invariant.
\end{theorem}

\begin{proof}
First of all, we begin from left to right. 
From the Definition \ref{inv}, if $\{B(\omega)\}$ is quasi-invariant, then
$$B(\theta_t\omega)\subset\widetilde{R}(t,\omega)B(\omega)\subset R(t,\omega)B(\omega),\forall t\geq0.$$
Therefore, if $\{B(\omega)\}$ is quasi-invariant and  positive invariant, then $\{B(\omega)\}$ is  invariant.

Conversely, if $B(\omega)$ is invariant, then $B(\omega)$ is also positive invariant. Next we show that it is quasi-invariant. Let $u_0\in B(\omega)$, we can find a random complete trajectory $u:\mathbb{R}\rightarrow X$ such that $u(0)=u_0$ and $u(t)\in B(\theta_t\omega)$ for all $t\in \mathbb{R}$. Indeed, for $t\geq0$, we simply set $u(t)=R(t,\omega)u_0$, which is in $B(\theta_t\omega)$, because $B(\omega)$ is invariant.

In order to construct $u(t)$ for $t<0$, we proceed inductively. Since $R(1,\omega)B(\omega)=B(\theta_1\omega)$, then there exists $u(-1)\in B(\theta_{-1}\omega)$ such that $R(0,\theta_{-1}\omega)u(\theta_{-1}\omega)=u(0)$. Let 
$$ u(t) = R(t,\theta_{-1}\omega)u(-1)  \mbox{~for~} -1\leq t<0.$$
Now, there exists  $u(-2)\in B(\theta_{-2}\omega)$ such that $R(-1,\theta_{-2}\omega)u(-2)=u(-1)$ and we set 
$$u(t)=R(t,\theta_{-1}\omega)u(-2) \mbox{~ for~} -2\leq t<-1.$$ This procedure yields the required complete trajectory $u(t),u(t)\in B(\theta_t\omega),t\in \mathbb{R}$. The proof is complete.
\end{proof}


%

$\mathcal{E}^{\theta}_{\omega}((-\infty,\infty))$ consists of all random complete trajectories. We will use random complete trajectory to characterize the structure of global random attractor $\{\mathcal{A}_w(\omega)\}$ and $\{\mathcal{A}_s(\omega)\}$.
Define $$\Xi_{\omega}(t)=\{u(t)|u\in \mathcal{E}^{\theta}_{\omega}((-\infty,\infty))\}$$ collection of value of random complete trajectories.

\begin{lemma}\label{l5.1}
Suppose that there exists a family of random absorbing set and the random evolutionary semigroup $\{R(t,\omega)\}$ is weak upper semicontinuous in $X$ for every $t\in \mathbb{R}^+, \omega\in\Omega$. Then the weak global random attractor $\{\mathcal{A}_w(\omega)\}$ satisfies $\mathcal{A}_w(\omega)=\Xi_{\omega}(0)$ for every $\omega\in\Omega$.
\end{lemma}
\begin{proof}
$\Xi_{\omega}(t)$ is invariant by definition. So we have $\Xi_{\omega}(0)\subset \mathcal{A}_w(\omega)$. If $u(0)\in \mathcal{A}_w(\omega)$, then $R(t,\omega)u(0)\subset \mathcal{A}_w(\theta_t\omega)$ for all $t\geq 0$. Since $\mathcal{A}_w(\omega)$ is invariant, we can use the inductive procedure like Theorem \ref{t5.2} to construct a random complete trajectory $u(t)$, $t\in \mathbb{R}$ satisfy $u(0)\in \mathcal{A}_w(\omega)$ . The proof is complete.
\end{proof}

Similar result also holds to strong global random attractor.

\begin{corollary}
Suppose that there exists a family of random absorbing set. If the random evolutionary semigroup $\{R(t,\omega)\}$ is weak upper semicontinuous in $X$ for every $t\in \mathbb{R}^+, \omega\in\Omega$ and $s-$asymptotically compact, then the strong global random attractor $\{\mathcal{A}_s(w)\}$ satisfies $\mathcal{A}_s(w)=\Xi_{\omega}(0)$.
\end{corollary}

The following results tell us the tracking property of weak and strong global random attractor. For arbitrary random trajectory of phase space, there is random trajectory start from weak and strong global random attractor, to appropriate former, hence all the information of whole random dynamical system are reflected by global random attractor.

\begin{theorem}(Weak Tracking Property)\label{t5.3}
Let $\mathcal{E}^{\theta}_{\omega}([T,\infty)),T\in \mathbb{R}^+$ be a random evolutionary system. Also suppose that $\mathcal{E}^{\theta}_{\omega}([T,\infty))$ is a compact set in $C([T,\infty);X_w)$. Let $A\subset X$. Then for any $\varepsilon>0$, there exists $t_0\geq0$, such that for any $t^*>t_0$, every random trajectory $u\in\mathcal{E}^{\theta}_{\omega}([t^*,\infty))$ with $u(t^*)\in A$ satisfies
$$d_{C([t^*,\infty);X_w)}(u,v)<\varepsilon,$$
for some complete trajectory $v\in \mathcal{E}^{\theta}_{\omega}((-\infty,\infty))$ with $v(t)\in \Omega_w(A,\omega)$ for all $t\in \mathbb{R}, \omega\in\Omega$.
\end{theorem}
\begin{proof}
Suppose that the claim is not true. Then there exist $\varepsilon>0$, $u_n\in \mathcal{E}^{\theta}_{\omega}([\tau_n,\infty))$ with $u_n(\tau_n)\in A$ and $\tau_n\geq 0$, $\tau_n\rightarrow\infty$ as $n\rightarrow\infty$ such that
$$d_{C([\tau_n,\infty);X_w)}(u_n,v)\geq\varepsilon,$$
for all $n$ and any $v\in \mathcal{E}^{\theta}_{\omega}((-\infty,\infty))$ with $v(t)\in \Omega_w(A,\omega)$, $t\in \mathbb{R}$.

On the other hand, consider a sequence $v_n\in \mathcal{E}^{\theta}_{\omega}([-\tau_n,\infty))$, $v_n(t)=u_n(t+\tau_n)$. Thanks to the fact that $\mathcal{E}^{\theta}_{\omega}([-\tau_n,\infty))$ is compact in $C([-\tau_n,\infty);X_w)$ for all n, as in functional diaganalization process method, by choose subsequence, we can infer that there exists $v\in \mathcal{E}^{\theta}_{\omega}((-\infty,\infty))$, such that $v_n\rightarrow v$ in $C([T,\infty);X_w)$ as $n\rightarrow\infty$ for all $T\in \mathbb{R}$, in particular $v(t)\in \Omega_w(A,\omega)$ for all $t\in \mathbb{R}$.
Hence, for large enough $n$, we have
$$d_{C([0,\infty);X_w)}(v_n,v)<\varepsilon,$$
thus
$$d_{C([\tau_n,\infty);X_w)}(u_n,v(\cdot-\tau_n))<\varepsilon,$$
which is a contradiction. The proof is complete.
\end{proof}

\begin{theorem}\label{sptp}(Strong Tracking Property)
 Suppose that random evolutionary system $\mathcal{E}^{\theta}_{\omega}([T,\infty))$, $T\in \mathbb{R}^+$ satisfies conditions in Theorem \ref{t5.3} and it is $s-$asymptotically compact. Let $A\subset X$. Then for any $\varepsilon>0$ and $T>0$, there exists $t_0\geq0$, such that for any $t^*>t_0$, every trajectory $u\in \mathcal{E}^{\theta}_{\omega}([t^*,\infty))$ with $u(t)\in A$ satisfies
$$d_s(u(t),v(t))<\varepsilon,\forall t\in[t^*,t^*+T],$$
for some complete trajectory $\mathcal{E}^{\theta}_{\omega}((-\infty,\infty))$ with $v(t)\in \Omega_s(A,\omega)$ for all $t\in \mathbb{R}$.
\end{theorem}
\begin{proof}
Suppose that the claim does not hold. Then there exist $\varepsilon>0$, $T>0$ and sequences $u_n\in \mathcal{E}^{\theta}_{\omega}([\tau_n,\infty))$ with $u_n(\tau_n)\in A$ and $\tau_n\geq 0,\tau_n\rightarrow\infty$ as $n\rightarrow\infty$ such that
$$\sup_{t\in[\tau_n,\tau_n+T]}d_s(u_n(t),v(t))\geq\varepsilon,\forall n,$$
for all $v\in \mathcal{E}^{\theta}_{\omega}((-\infty,\infty))$ with $v(t)\in \Omega_s(A,\omega),\forall t\in \mathbb{R}$.

On the other hand, according to assumption, there exists a sequence $v_n\in \mathcal{E}^{\theta}_{\omega}((-\infty,\infty))$ with $v_n(t)=u_n(t)$ on $[\tau_n,\infty)$ such that
$$\lim_{n\rightarrow\infty}\sup_{t\in[\tau_n,\tau_n+T]}d_s(v_n(t),v(t))=0,$$ which is a contradiction. The proof is complete.
\end{proof}

\section{3D   Navier-Stokes equation with nonlinear colored noise}
Given a metric dynamical system $(\Omega,  \mathcal{F},    \mathbb{P}, \{\theta_t\}_{t\in \mathbb{R}})$, where
 $$\Omega=\{\omega\in C(\mathbb{R},\mathbb{R}):\, \omega(0)=0\}$$
equipped with the compact-open topology, $\mathcal{F}$ is the Borel sigma-algebra $\mathcal{B}(\Omega)$, $\mathbb{P}$ is the Wiener measure and $\{\theta_t\}_{t\in \mathbb{R}}$ is the measure-preserving transformation group on $\Omega$ by
$$
\theta_t\omega(\cdot)=\omega(\cdot+t)-\omega(t), \quad \omega\in\Omega,\,t\in\mathbb{R}.
$$
Consider the following 3D random Navier-Stokes equations driven by colored noise:
\begin{eqnarray}\label{NS-clor}
\left\{\begin{array}{llc}
\frac{\partial u}{\partial t}=\nu\triangle u -u\cdot\nabla u-\nabla p+f+G(x,u)\zeta_{\delta}(\theta_t\omega),\, ~x\in D, ~t>\tau, \\
\mbox{div}\, u=0,  \\
\end{array}\right.
\end{eqnarray}
supplemented by the non-slip  boundary and initial conditions
\begin{eqnarray}\label{NS-1clorb}
\left\{\begin{array}{llc}
u|_{\partial D}=0,\\
u|_{t=\tau}=u_\tau,  \\
\end{array}\right.
\end{eqnarray}
where $\tau\in \mathbb{R}$ and the colored noise is a random variable defined by
$$
\zeta_{\delta}(\theta_t \omega)=-\frac{1}{\delta^2}\int^0_{-\infty} e^{\frac{s}{\delta}}\theta_t \omega(s) ds=\frac{1}{\delta}\int^t_{-\infty} e^{\frac{1}{\delta}(s-t)}  dW(s,\omega)
$$
with $W(t,\omega)=\omega(t)$ is a two-sided real-values Wiener process defined on $(\Omega,  \mathcal{F},    \mathbb{P})$ for $t\in \mathbb{R}$ and $\omega\in \Omega$. For simplicity, the external force term is considered to be autonomous with respect to time.

In this section, we use our abstract results that established in above sections to analyze the dynamical behavior of stochastic system \eqref{NS-clor}-\eqref{NS-1clorb}.

\subsection{Preliminary}
\quad Firstly, let $D\subset\mathbb{R}^3$ be a bounded domain with smooth boundary $\partial D$.  Denote $$\mathcal{V}=\{u\in (C_{0}^{\infty}(D))^{3}|~\mbox{div} u=0, \int_{D}udx=0\}, $$ $ \boldsymbol L^{2}(D) =(L^{2}(D))^{3}$ and $ \boldsymbol{H}_{0}^{1}(D)=(H_{0}^{1}(D))^{3} $. $ H $ is  the closure of $ \mathcal{V}$ in $\boldsymbol{L}^{2}(D)  $  with the inner product $ (u,v) $, where for $ u $, $ v \in H $
\begin{equation*}
	\begin{aligned}
		(u,v)=\sum\limits_{\scriptstyle i=1}^3\int_{D}u_i(x)v_i(x)dx,
	\end{aligned}
\end{equation*}
which leads to the norm as $\|u\|^2:=(u,u)$.
$V$  is the closure of $ \mathcal{V}$ in $ \boldsymbol{H}_{0}^{1} (D)$ with the inner product $ ((u,v))$, where for $u,~v\in V $
 \begin{equation*}
	\begin{aligned}
		((u,v))=\sum\limits_{\scriptstyle i,j=1}^3\int_{D} \partial_i u_j \partial_i v_j dx,
	\end{aligned}
\end{equation*}
where $\partial_i=\frac{\partial}{\partial x_i}$, which leads to the norm as  $\|u\|_V^2:=((u,u))$.
Let $ H^{*} $ and $V^{*} $ be the dual spaces of  $H$ and $V$  respectively.
\par
Secondly, we define a linear operator $ A:~V\to V^{*} $ as
$$ \langle Au,~v\rangle=((u,v))=(\nabla u,~\nabla v),~ \forall ~u, v\in V. $$
Let $\mathfrak{D}(A) = \{u\in V,~Au\in H\}.$  We define $ V^{\alpha}=\mathfrak{D}(A^{\frac{1}{2}})$ with the inner product
$$\langle u,v \rangle_{V^{\alpha}}=( u,v )+\alpha^{2}( A^{\frac{1}{2}}u,A^{\frac{1}{2}}v).$$
Furthermore,
$$A^{-1} :~ H\to \mathfrak{D}(A).$$
Because $ A^{-1}$ is self-adjoint, compact and positive definite, we have a set of standard orthogonal basis $\{w_{j}\}_{j=1}^{\infty}$ of $H$, that is
$$ A^{-1}w_{j} =\mu _{j}w_{j},$$
and $ w_{j}=\frac{1}{\mu _{j}}A^{-1}w_{j}\in \mathfrak{D}(A) $. Let $\lambda _{j}=\frac{1}{\mu_{j}}$, then
$ Aw_{j}=\lambda_{j}w_{j} $
and $ 0<\lambda_{1} \le \lambda_{2} \le \lambda_{3} \le \dots$, which means that $\{w_{j}\},~\{\lambda_{j}\} $ are the eigenvectors and eigenvalues of $ A $ respectively.
Furthermore, the embedded relations are true: $ \mathfrak{D}(A)\hookrightarrow V\hookrightarrow H\equiv H ^*\hookrightarrow V^* $. Notice that $\{w_{j}\}_{j=1}^{\infty}$ is orthogonal in $V$, since for all $ j \neq k$,  $$ ((w_{j},w_{k}))=\langle Aw_{j},w_{k}\rangle=(\lambda_{j}w_{j}, w_{k})=\lambda_{j}(w_{j},w_{k})=0,$$
$$ \langle Aw_{j},w_{j}\rangle=\lambda_{j}\|w_{j}\|^{2}=\lambda_{j}.$$

In fact, the norm of $H$   can be defined by  eigenvectors.
If $u\in H$, then $u=\sum\limits_{j=1}^{\infty}u_{j}w_{j}$, where $ (u_{1}, u_{2}, \ldots,u_{j},\ldots)$ is called the coordinate of $ u $
in $ H $ and the norm is
$$ \|u\|^{2} = \sum_{j=1}^{\infty}|u_{j}|^{2}\|w_{j}\|^{2}= \sum_{j=1}^{\infty}|u_{j}|^{2}.$$
In addition, we define the following metric in $ H $,
\begin{equation} \label{dwh}
	d_{w}(u,v)=\sum_{j=1}^{\infty}\dfrac{1}{2^{j}} \frac{|u_{j}-v_{j}|}{1+|u_{j}-v_{j}|},~~~~~~u,v\in H ,
 \end{equation}
which is not a norm.

Next, for all $u, v, w\in V$, the  bilinear and trilinear operators are defined by
\begin{eqnarray}\nonumber
a(u,v)=(\nabla u, \nabla v),\quad\quad\quad\quad \quad\quad\\
b(u,v,w)=\langle (u \cdot \nabla )v,w\rangle=\sum\limits_{\scriptstyle i,j=1}^2\int_\Omega u_i (\partial_i v_j) w_j dx,
\end{eqnarray}
and a bilinear continuous operator of $B: V\times V \rightarrow V^{*}$ is introduced by
\begin{eqnarray}\nonumber
\langle B(u,v), w\rangle =b(u,v,w).
\end{eqnarray}

\subsection{Formulation and existence of weak solution}
Given a metric dynamical system $(\Omega,  \mathcal{F},    \mathbb{P}, \{\theta_t\}_{t\in \mathbb{R}})$, where
 $$\Omega=\{\omega\in C(\mathbb{R},\mathbb{R}):\, \omega(0)=0\}$$
equipped with the compact-open topology, $\mathcal{F}$ is the Borel sigma-algebra $\mathcal{B}(\Omega)$, $\mathbb{P}$ is the Wiener measure and $\{\theta_t\}_{t\in \mathbb{R}}$ is the measure-preserving transformation group on $\Omega$ by
$$
\theta_t\omega(\cdot)=\omega(\cdot+t)-\omega(t), \quad \omega\in\Omega,\,t\in\mathbb{R}.
$$
Consider the following 3D random Navier-Stokes equations driven by colored noise:
\begin{eqnarray}\label{NS-clor}
\left\{\begin{array}{llc}
\frac{\partial u}{\partial t}=\nu\triangle u -u\cdot\nabla u-\nabla p+f+G(x,u)\zeta_{\delta}(\theta_t\omega),\, ~x\in D, ~t>\tau, \\
\mbox{div}\, u=0,  \\
\end{array}\right.
\end{eqnarray}
supplemented by the non-slip  boundary and initial conditions
\begin{eqnarray}\label{NS-1clorb}
\left\{\begin{array}{llc}
u|_{\partial D}=0,\\
u|_{t=\tau}=u_\tau,  \\
\end{array}\right.
\end{eqnarray}
where $\tau\in \mathbb{R}$ and the colored noise is a random variable defined by
$$
\zeta_{\delta}(\theta_t \omega)=-\frac{1}{\delta^2}\int^0_{-\infty} e^{\frac{s}{\delta}}\theta_t \omega(s) ds=\frac{1}{\delta}\int^t_{-\infty} e^{\frac{1}{\delta}(s-t)}  dW(s,\omega)
$$
with $W(t,\omega)=\omega(t)$ is a two-sided real-values Wiener process defined on $(\Omega,  \mathcal{F},    \mathbb{P})$ for $t\in \mathbb{R}$ and $\omega\in \Omega$. For simplicity, the external force term is considered to be autonomous with respect to time.

The concept of weak solution to system \eqref{NS-clor}-\eqref{NS-1clorb} is defined firstly.
\begin{definition}
Let $f\in V^\ast$, given the filtered probability space $(\Omega,  \mathcal{F}, \mathcal{F}_t,  \mathbb{P})$, $\omega\in \Omega$, $\tau\in\mathbb{R}$ and  initial value $u_\tau\in H$, a function $u(\cdot,\tau, \omega,u_\tau): [\tau,\infty)\rightarrow H$ is called a weak solution of system \eqref{NS-clor}-\eqref{NS-1clorb}, if for every $T>\tau$,
\begin{enumerate}
\item[(1)] $  u\in C([\tau,T];H_w)\cap L^{\infty}(\tau,T;H)\cap L^2(\tau,T;V)$;
\item[(2)]for all $\phi\in V$ and $t>\tau$, it satisfies
\begin{eqnarray}\nonumber
(  u(t), \phi)-\int_\tau^tb(u(s),\phi,u(s))ds&=&( u_\tau, \phi)+\nu\int_\tau^t\langle  u(s), A\phi\rangle ds+\int_\tau^t\langle f, \phi\rangle ds\\
&& +\int_\tau^t\int_{D} G(x,u(s))\phi(x)\zeta_{\delta}(\theta_s\omega) dx ds,
\end{eqnarray}
in the sense of distribution on $[\tau, \infty)$.
\end{enumerate}
\end{definition}

Throughout the paper, assume that the nonlinear diffusion term $G(x,u): D\times \mathbb{R}^3 \rightarrow \mathbb{R}^3 $ is a continuous function satisfying
\begin{eqnarray}\label{noncond}
|G(x,u)|\leq \eta_1(x)|u|^p+\eta_2(x),
\end{eqnarray}
where $0\leq p< 1$, $\eta_1\in  L^{\frac{2}{1-p}}(D) $ and $\eta_2\in L^2(D)$.
\begin{theorem}\label{exithe}
Assume $f\in  V^\ast$ and the initial value $u_\tau\in H$, then for every  $\omega\in \Omega$,  there exists at least one weak solution to \eqref{NS-clor}-\eqref{NS-1clorb} on $[\tau, \infty)$ with $u(\tau)=u_\tau$, which satisfies the following energy inequality
\begin{eqnarray}\label{energine}
\frac{1}{2} \frac{d}{dt}\|u(t)\|^2+ \nu  \|u(t)\|_V^2 \leq   \langle f, u(t)\rangle  +\zeta_{\delta}(\theta_t\omega)\big(G(x,u(t)), u(t)\big), \mbox{~~for~} t\in [\tau, T],
\end{eqnarray}
	for every $T\geq \tau$ in the distribution sense. Furthermore, there exists a constant $C$ such that
\begin{eqnarray}\label{energinese}\nonumber
\|u(t)\|^2 &\leq& e^ {-\frac{\nu\lambda_1(t-\tau)}{2}}\|u(\tau)\|^2 +C \int_\tau^te^ {-\frac{\nu\lambda_1(t-s)}{2}}\|u(\tau)\|^2\big(\|f\|^2_{V^\ast}\\
&&+  |\zeta_{\delta}(\theta_s\omega)|^{\frac{2}{1-p}}\|\eta_1\|^{\frac{2}{1-p}}_{L^{\frac{2}{1-p}}}+  |\zeta_{\delta}(\theta_s\omega)|^{2}\|\eta_2\|^2_{L^2} \big) ds  ,
\end{eqnarray}
 for all $t \geq \tau$.
\end{theorem}
\begin{proof}
{\bf Step 1:} A local approximated solution of \eqref{NS-clor}-\eqref{NS-1clorb} is constructed.

 Let $\{w_{j}\}_{\scriptstyle j=1}^\infty$ be an orthonormal basis of $H$ and  $W_{m}=$span\{$w_{1},\cdots,w_{m}$\} be subspace of $V$. The projector $P_{m}: H\rightarrow V_m$ is given by
 $$
 P_{m}v=\sum\limits_{\scriptstyle j=1}^m(v,w_{j})w_{j} \mbox{~for~all~} v\in H.
 $$
Next, we construct the approximated solution
$$u_{m}=\sum\limits_{\scriptstyle j=1}^m  g_{j,m}(t)w_{j}$$
satisfying the following equations
\begin{eqnarray}
\left\{\begin{array}{ll}
(u'_m(t), w_j)+\nu (\nabla u_m(t), \nabla w_j)+b(u_m(t),u_m(t),w_j)
=(f, w_j)+\zeta_{\delta}(\theta_t\omega)(G(x,u_m(t)),w_j),\\
u_{m}(\tau)=P_{m}u_\tau,   \\
\end{array}\right.\label{ode1}
\end{eqnarray}
for all $1\leq j\leq m$.
The above Cauchy problem   is a well-known ordinary functional differential equations with respect to the
unknown variables $\{g_{j,m}(t)\}_{j=1}^m$, which has a unique local solution (in an interval
$[\tau, t^*]$ with $\tau<t^*\leq T)$ by the local existence of solution for ordinary differential equations.

{\bf Step 2:}  The {\it a priori} estimates for $\{u_{m}\}$ is  derived.

Multiplying the first equation in \eqref{ode1} by $g_{j,m}(t)$ and summing in $j$, one obtains that
   \begin{eqnarray}   \label{estm1}\nonumber
  \frac{d}{dt}\|u_{m}\|^2+2\nu\| u_{m}\|_V^2&=&2(f,u_{m})+2\zeta_{\delta}(\theta_t\omega)(G(x,u_m), u_m)\\
  &\leq&\frac{4}{\nu }\|f\|_{V^\ast}^2+ \nu\|u_{m}\|_{V}^2+2|\zeta_{\delta}(\theta_t\omega)||(G(x,u_m), u_m)|.
 \end{eqnarray}
From \eqref{noncond} and the pathwise continuity of $\zeta_{\delta}(\theta_t\omega)$, one has
 \begin{eqnarray} \label{enr12}
 |\zeta_{\delta}(\theta_t\omega)||(G(x,u_m), u_m)|\leq C \big( \|\eta_1\|^{\frac{2}{1-p}}_{L^{\frac{2}{1-p}}(D)}+ \|\eta_2\|^2_{L^2(D)}+\|u_{m}\|^2\big),
 \end{eqnarray}
for all $\tau<t\leq T$. By plugging \eqref{enr12} into \eqref{estm1} and using Gronwall's inequality, one has
 \begin{eqnarray} \label{enr13}
\{u_m\}^\infty_{m=1} \mbox{~is~bounded~in~} L^{\infty}(\tau,T;H)\cap L^2(\tau,T;V).
 \end{eqnarray}
 Since
\begin{eqnarray} \label{enr14}
\frac{\partial u_m}{\partial t}=\nu\triangle u_m -P[u_m \cdot\nabla u_m]+P f+ P G(x,u_m)\zeta_{\delta}(\theta_t\omega),
 \end{eqnarray}
where $P$ is the Helmholz-Leray orthogonal projection in $\boldsymbol L^{2}(D)$ onto the space $H$. Simple computation shows that
\begin{eqnarray}\label{unimf33}
\|G(x,u_m)\|_{L^{\infty}(\tau,T; \boldsymbol L^{2}(D))}\leq \|\eta_1\|_{ L^{\frac{2}{1-p}}(D)}\|u_m\|^{p}_{L^{\infty}(\tau,T;H)}+\|\eta_2\|_{ L^{2}(D)}\leq C.
\end{eqnarray}
Then, one obtains that
\begin{eqnarray}\label{unifrom2}
    \Big\{\frac{\partial u_{m}}{\partial t}\Big\} \mbox{~is~bounded~in~} L^{\frac{4}{3}}(\tau,T;V').
     \end{eqnarray}

{\bf Step 3:} The limit is passed to derive the global solution by compact argument.

Form the Aubin-Lions Lemma, the following space
$$
W=\big\{u|~~u\in L^{2}(\tau,T;V): \frac{\partial u}{\partial t}\in L^{\frac{4}{3}}(\tau,T; V')\big\}
$$
is compactly embedded in $L^{2}(\tau,T; H)$.  Combining the preceding uniform estimates \eqref{enr13} and \eqref{unifrom2}, we
can deduce that there exists a subsequence $u_{m}$ (without relabeling) such that, when $m\rightarrow \infty$,
\begin{eqnarray}\label{comconver1}
&&u_{m}\rightarrow u \mbox{~stongly~in~} L^2\big(\tau,T; H\big) \mbox{~and~a.e.~};\\ \label{comconver2}
&&u_{m}\rightharpoonup  u \mbox {~weakly~}\ast \mbox {~in~} L^{\infty}(\tau,T;H);\\ \label{comconver4}
&&u_{m}\rightharpoonup u \mbox{~weakly~in~} L^{2}(\tau,T;V);\\ \label{comconver5}
&&\frac{\partial u_{m}}{\partial t} \rightharpoonup \frac{\partial u}{\partial t} \mbox{~weakly~in~} L^{\frac{4}{3}}(\tau,T;V'),
\end{eqnarray}
 with $u\in C([\tau,T];H_w)\cap L^{\infty}(\tau,T;H)\cap L^2(\tau,T;V)$.  By the continuity of $G$, one also has
$$
 G(x,u_m)  \rightarrow  G(x,u)~~ \mbox{~a.e.~},
$$
by combining \eqref{unimf33} and using Lions' Lemma 1.3 in chapter 1 of \cite{L1969}, we derive that
\begin{eqnarray}\label{nonconv}
  G(x,u_m)\rightharpoonup   G(x,u) \mbox {~weakly~}\ast \mbox{~in~} L^{\infty}(\tau,T; \boldsymbol L^{2}(D)).
\end{eqnarray}
Thus, passing to the limit of \eqref{ode1}, we conclude that $u$ is exactly a weak solution to \eqref{NS-clor}-\eqref{NS-1clorb}.

{\bf Step 4:} Next, the inequality \eqref{energine} is established.

Let $ \theta \in C^{\infty}_{0}([\tau,T])$ be a positive function. From \eqref{estm1}, one has
\begin{eqnarray}   \label{estmenin}\nonumber
 && -\int^{T}_{\tau}\theta'(s) \|u_{m}(s)\|^2ds+2\nu\liminf_{m\to\infty}(\int^{T}_{\tau}\theta(s)\| u_{m}(s)\|_V^2ds)\\
  &=&2\int^{T}_{\tau}\theta(s) (f,u_{m}(s))ds+2\int^{T}_{\tau} \theta(s)\zeta_{\delta}(\theta_s\omega)(G(x,u_m(s)), u_m(s))ds.
 \end{eqnarray}
Combining \eqref{comconver1}, \eqref{comconver4} and \eqref{nonconv}, therefore
\begin{eqnarray}   \label{estmeninf}\nonumber
&& -\int^{T}_{\tau}\theta'(s) \|u(s)\|^2ds+2\nu\int^{T}_{\tau}\theta(s)\| u(s)\|_V^2ds\\
 &\leq& \liminf_{m\to\infty}\Big(-\int^{T}_{\tau}\theta'(s) \|u_{m}(s)\|^2ds+2\nu\int^{T}_{\tau}\theta(s)\| u_{m}(s)\|_V^2ds\Big)\\
  &\leq&2\int^{T}_{\tau}\theta(s) (f,u(s))ds+2\int^{T}_{\tau} \theta(s)\zeta_{\delta}(\theta_s\omega)(G(x,u(s)), u(s))ds,
 \end{eqnarray}
which means that \eqref{energine} satisfies in the distribution sense.

{\bf Step 5:}  Finally, the inequality \eqref{energinese} is established.

From \eqref{estm1},  one has
\begin{eqnarray}   \label{energinese1}\nonumber
  \frac{d}{dt}\|u_{m}\|^2+\nu\| u_{m}\|_V^2  \leq\frac{4}{\nu }\|f\|_{V^\ast}^2+2|\zeta_{\delta}(\theta_t\omega)|\big(\|\eta_1\|_{L^{\frac{2}{1-p}}}\|u_{m}\|^{p+1}+ \|\eta_2\|_{L^2}\|u_{m}\|   \big)
 \end{eqnarray}
By Poincar\'{e} inequality and Young's inequality, it follows that
\begin{eqnarray}   \label{energinese2}\nonumber
  \frac{d}{dt}\|u_{m}\|^2+\frac{\nu\lambda_1}{2}\| u_{m}\|^2  \leq \frac{4}{\nu }\|f\|_{V^\ast}^2+C\big( |\zeta_{\delta}(\theta_t\omega)|^{\frac{2}{1-p}}\|\eta_1\|^{\frac{2}{1-p}}_{L^{\frac{2}{1-p}}}+ |\zeta_{\delta}(\theta_t\omega)|^{2}\|\eta_2\|^2_{L^2} \big)
 \end{eqnarray}
Using Gronwall’s inequality, we obtain that
\begin{eqnarray}\label{energinese3}\nonumber
\|u_m(t)\|^2 &\leq& e^ {-\frac{\nu\lambda_1}{2}t}\Big(e^ {\frac{\nu\lambda_1}{2}\tau}\|u_m(\tau)\|^2 +C \int_\tau^t e^ {\frac{\nu\lambda_1}{2}s}\big(\|f\|^2_{V^\ast}\\
&&+  |\zeta_{\delta}(\theta_s\omega)|^{\frac{2}{1-p}}\|\eta_1\|^{\frac{2}{1-p}}_{L^{\frac{2}{1-p}}}+  |\zeta_{\delta}(\theta_s\omega)|^{2}\|\eta_2\|^2_{L^2}  \big) ds  \Big),
\end{eqnarray}
which implies \eqref{energinese} by passing to limit.
\end{proof}

\subsection{Absorbing set and weakly upper semicontinuity}
\begin{lemma}\label{attrset1} 
Given initial data $ u_{\tau} \in  H $, $f\in V^*$, $\eta_1\in  L^{\frac{2}{1-p}}(D)$ and $\eta_2\in L^{2}(D)$, 	 then the system \eqref{NS-clor}-\eqref{NS-1clorb} has an absorbing set
	\begin{eqnarray}\label{absorbingset}
		\mathbb{Y}(\omega)=\big\{u\in H,~\|u\|^{2} \le R^{2}(\omega)\big \},
	\end{eqnarray}
	with  $R^{2}(\omega)$ is defined by 
\begin{eqnarray}\label{aborid}
R^{2}(\omega)= C \big(\|\eta_1\|^{\frac{2}{1-p}}_{L^{\frac{2}{1-p}}}+\|\eta_2\|^{2}_{L^{2}}+\|f\|^{2}_{V^\ast}\big),
\end{eqnarray}
where $C$ is a constant depending only on $p$, $\omega$, $\nu$ and  $\lambda_1$.
\end{lemma}
\begin{proof}
	By \eqref{energinese},   one has 
	\begin{eqnarray}\label{energiabs}\nonumber
\|u(t,\theta_{-t} \omega, u_0)\|^2 &\leq& e^ {-\frac{\nu\lambda_1}{2}t}\Big(\|u_0\|^2 +C \int_0^t e^ {\frac{\nu\lambda_1}{2}s}\big(\|f\|^2_{V^\ast}\\
&&+  |\zeta_{\delta}(\theta_{s-t}\omega)|^{\frac{2}{1-p}}\|\eta_1\|^{\frac{2}{1-p}}_{L^{\frac{2}{1-p}}}+  |\zeta_{\delta}(\theta_{s-t}\omega)|^{2}\|\eta_2\|^2_{L^2} \big) ds  \Big),
\end{eqnarray}
 for all $t \geq 0$.

Firstly, it is easy to get
\begin{eqnarray} \label{omeest13}
\int_0^t e^ {-\frac{\nu\lambda_1}{2}(t-s)} \|f\|^2_{V^\ast}ds \leq C_1  \|f\|^2_{V^\ast}.\quad \quad\quad
	\end{eqnarray}
where $C_1$ is a constant depending only on $\nu$ and  $\lambda_1$.

	From Lemma 2.9 in \cite{GW2020}, there exists a $\{\theta_t\}_{t\in \mathbb{R}}$-invariant subset (still denoted by) $\Omega$ of  full measure, such that for all $\omega \in \Omega$,
	\begin{equation}\label{omegare}
		 \lim\limits_{t\rightarrow \pm \infty} \frac{|\zeta_{\delta}(\theta_{t}\omega)|}{t}=0 \quad \mbox{~for~every~} \delta\in (0, 1].
	\end{equation}
	Thus, for every $\omega \in \Omega$, there exists a constant  $M>1$ depending on $\omega$ such that
\begin{eqnarray}\label{omeest}
|\zeta_{\delta}(\theta_{t}\omega)| \leq M \mbox{~for~all~} t\in  \mathbb{R}.
\end{eqnarray}
	Therefore, for all $t\geq0$,
	\begin{eqnarray}\label{omeest11}
		 \int_0^t e^ {-\frac{\nu\lambda_1}{2}(t-s)} |\zeta_{\delta}(\theta_{s-t}\omega)|^{\frac{2}{1-p}}\|\eta_1\|^{\frac{2}{1-p}}_{L^{\frac{2}{1-p}}}ds
\leq M^{\frac{2}{1-p}} \|\eta_1\|^{\frac{2}{1-p}}_{L^{\frac{2}{1-p}}}\int_0^t e^ {-\frac{\nu\lambda_1}{2}(t-s)} ds 
\leq C_2 \|\eta_1\|^{\frac{2}{1-p}}_{L^{\frac{2}{1-p}}},
	\end{eqnarray}
here $C_2$ is a constant depending only on $p$, $\omega$, $\nu$ and  $\lambda_1$.

Similarly,  one also has
	\begin{eqnarray}\label{omeest12}
	 \int_0^t e^ {-\frac{\nu\lambda_1}{2}(t-s)} |\zeta_{\delta}(\theta_{s-t}\omega)|^{2}\|\eta_2\|^2_{L^2} ds
\leq C_3\|\eta_2\|^2_{L^2}, 
	\end{eqnarray}
here  $C_3$ is a constant depending only on $p$, $\omega$, $\nu$ and  $\lambda_1$.

Combining \eqref{energiabs}, \eqref{omeest11} and \eqref{omeest12}, we have
	\begin{eqnarray}\label{energiabszui}
\|u(t,\theta_{-t} \omega, u_0)\|^2 &\leq& e^ {-\frac{\nu\lambda_1}{2}t} \|u_0\|^2 + C  \big(\|\eta_1\|^{\frac{2}{1-p}}_{L^{\frac{2}{1-p}}}+\|\eta_2\|^{2}_{L^{2}}+\|f\|^{2}_{V^\ast}\big),
\end{eqnarray}
here $C$ is a constant depending only on $p$, $\omega$, $\nu$ and  $\lambda_1$. By	taking $R^{2}(\omega)$ as \eqref{aborid}, the absorbing set  $ \mathbb{Y}(\omega)$ defined by \eqref{absorbingset} for system \eqref{NS-clor}-\eqref{NS-1clorb} is obtained.
\end{proof}
For every $m\in  \mathbb{N}$, define
\begin{eqnarray}
\Omega_m=\big\{\omega\in \Omega| ~|\omega(t)|\leq |t|\mbox{~~for~all~} |t|\geq m~\big\}.
\end{eqnarray}
Next, we prove the weak upper semicontinuity of the weak solution to system  \eqref{NS-clor}-\eqref{NS-1clorb}.
\begin{proposition}\label{sem-con}
Given $m\in  \mathbb{N}$, $\tau, \hat{t}\in \mathbb{R}$ and $\omega,\omega_n\in \Omega_m$, suppose $u(\cdot,\omega_n,u_{\tau,n})$ is a solution of system \eqref{NS-clor}-\eqref{NS-1clorb} with initial data $u_{\tau,n}$. If $t_n\rightarrow \hat{t}$, $\omega_n\rightarrow \omega$ and
$u_{\tau,n}\rightharpoonup u_\tau$ in $H$, then there exists a subsequence  (without relabeling) such that
\begin{eqnarray}\label{weauppes}
u(t_n,\omega_n,u_{\tau,n}) \rightharpoonup u(\hat{t},\omega,u_\tau) \mbox{~~in~} H,
\end{eqnarray}
where $u(\cdot,\omega,u_\tau)$ is a weak solution of  problem \eqref{NS-clor}-\eqref{NS-1clorb} with initial data $u_\tau$.
\end{proposition}
\begin{proof}
For convenience,  $u(t,\omega_n,u_{\tau,n})$ is abbreviated as $u^n(t)$.  From \eqref{energine}, one has
\begin{eqnarray}\label{unifest1}
&&\|u^n(t)\|^2+2\nu \int_{\tau}^t\|u^n(s)\|_V^2 ds\\ \nonumber
&\leq& \|u_{\tau,n}\|^2+2\int_{\tau}^t\langle f, u^n(s)\rangle ds
 +2\int_{\tau}^t \int_D G(x,u^n(s))u^n(s) \zeta_{\delta}(\theta_s\omega_n) dx ds,
\end{eqnarray}
for all $t\in [\tau, \hat{t}+1]$.
\begin{eqnarray}\label{unifest3}
\big|\int_{\tau}^t\langle f, u^n(s)\rangle ds \big|\leq \int_{\tau}^t\big(\| f\|^2+ \| u^n(s)\|^2 \big) ds.
\end{eqnarray}
Since $u_{\tau,n}\rightharpoonup u_\tau$ in $H$, there exists constant $C$ such that
\begin{eqnarray}\label{unifest2}
\sup\limits_{n\in \mathbb{N}}\|u_{\tau,n}\| \leq C.
\end{eqnarray}
By Lemma 3.3 in \cite{GW2020}, there exist  $N$ and $C$  depending on $\delta, \hat{t}, \tau, \omega$ such that
\begin{eqnarray}\label{unifest4}\nonumber
\sup\limits_{n\geq N}\sup\limits_{s\in  [\tau, \hat{t}+1]}|\zeta_{\delta}(\theta_s\omega_n) |\leq \sup\limits_{s\in  [\tau, \hat{t}+1]} |\zeta_{\delta}(\theta_s\omega)|+1 \leq C.
\end{eqnarray}
Then, one has
\begin{eqnarray}\label{unifest5}\nonumber
&&\Big|\int_{\tau}^t \int_D G(x,u^n(s))u^n(s) \zeta_{\delta}(\theta_s\omega_n) dx ds\Big| \\ \nonumber
&\leq& C \int_{\tau}^t \int_D |G(x,u^n(s))||u^n(s)|dx ds\\  \nonumber
&\leq& C\int_{\tau}^t \big( \|\eta_1\|^{\frac{2}{1-p}}_{L^{\frac{2}{1-p}}(D)}+ \|\eta_2\|^2_{L^2(D)}+\|u^n(s)\|^2\big)ds\\
&\leq& C\Big(1+\int_{\tau}^t  \|u^n(s)\|^2 ds \Big).
\end{eqnarray}
By combining \eqref{unifest1}-\eqref{unifest5} and using Gronwall’s inequality, one has
\begin{eqnarray}\label{unifest6}
\sup\limits_{t\in  [\tau, \hat{t}+1]}\|u^n(t)\|^2+2\nu \int_{\tau}^{\hat{t}+1}\|u^n(s)\|_V^2 ds\leq C.
\end{eqnarray}
Similar with the estimate of \eqref{unifrom2}, we also have
\begin{eqnarray}\label{unifest8}
 \Big\|\frac{\partial u^n}{\partial t}\Big\|^2_{L^{\frac{4}{3}}(\tau,\hat{t}+1;V')}\leq C.
\end{eqnarray}
Immediately, from \eqref{unifest6}, we have
 \begin{eqnarray}\label{unifest7}
 u^n (t_n)  \rightharpoonup u_{t_0} \mbox{~weakly~in~} H.
\end{eqnarray}
Furthermore, for every fixed $\phi\in V$ and all $t, s \in  [\tau, \hat{t}+1]$,
\begin{eqnarray} \nonumber
\big(u^n(t)-u^n(s), \phi\big)&=&\int_s^tb(u^n(\xi),\phi,u^n(\xi))d\xi+\nu\int_s^t\langle u^n(\xi), A\phi\rangle d\xi+\int_s^t\langle f, \phi\rangle d\xi\\
&& +\int_s^t\int_{D} G(x,u^n(\xi))\phi(x)\zeta_{\delta}(\theta_s\omega_n) dx d\xi,
\end{eqnarray}
then, by using \eqref{unifest6},
\begin{eqnarray} \nonumber
\big|\big(u^n(t)-u^n(s), \phi\big)\big|&\leq&C \|\phi\|_{V}\Big(\int_s^t \|u^n(\xi)\|\|u^n(\xi)\|_{V}d\xi+\nu\int_s^t \| u^n(\xi)\|_{V}  d\xi+\int_s^t \| f\|_{V^*}   d\xi\\ \nonumber
&& +\int_s^t \big( \|u^n(\xi)\|^p\|\eta_1\|_{L^{\frac{2}{1-p}}} +\|\eta_2\|_{L^2}\big)d\xi\Big)\\
&\leq&C \|\phi\|_{V} |t-s|^{\frac{1}{2}}.
\end{eqnarray}
which means that the subsequence $\{u^n\}_{n=1}^\infty$  is equicontinuous on $ [\tau, \hat{t}+1]$ in $ V^*$.
 Combining the fact
 $$\|u^n\|_{C( [\tau, \hat{t}+1]; V^*)}\leq C \|u^n\|_{W^{1,\frac{4}{3}}( \tau, \hat{t}+1; V^*)}\leq C,$$
then by the Ascoli-Arzel\`{a} theorem, one obtains that, up to a subsequence
 \begin{eqnarray}\label{unifest777}
 u^n    \rightharpoonup   u  \mbox{ ~in~}  C( [\tau, \hat{t}+1]; V^*).
\end{eqnarray}
On the other hand, combining \eqref{unifest6} and \eqref{unifest8}, it is easy to get that, up to a subsequence
\begin{eqnarray}\label{comconver222}
&&u^n  \rightarrow u  \mbox{~stongly~in~} L^2\big(\tau, \hat{t}+1; H\big) \mbox{~and~a.e.~};\\ \label{comconver23}
&&u^n  \rightharpoonup  u \mbox {~weakly~}\ast \mbox {~in~} L^{\infty}(\tau,\hat{t}+1;H);\\ \label{comconver43}
&&u^n \rightharpoonup u \mbox{~weakly~in~} L^{2}(\tau,\hat{t}+1;V);\\ \label{comconver53}
&&\frac{\partial u^n}{\partial t} \rightharpoonup \frac{\partial u}{\partial t} \mbox{~weakly~in~} L^{\frac{4}{3}}(\tau,\hat{t}+1;V'),
\end{eqnarray}
 with $u=u(\cdot,\omega,u_{\tau})\in C([\tau,\hat{t}+1];H_w)\cap L^{\infty}(\tau,\hat{t}+1;H)\cap L^2(\tau,\hat{t}+1;V)$, which is a weak solution of  problem \eqref{NS-clor}-\eqref{NS-1clorb} with initial data $u_\tau$.
Therefore, by combining \eqref{unifest7}, \eqref{unifest777} and \eqref{comconver222}, one finishes the proof of \eqref{weauppes}.
\end{proof}

\subsection{$s-$asymptotic compactness and global random attractor}
By using the weak solutions of \eqref{NS-clor}-\eqref{NS-1clorb}, one can construct the following random  evolutionary system $\mathcal{E}^\theta_\omega$,
\begin{eqnarray}\nonumber
		\mathcal{E}^\theta_\omega([T,\infty))=&\big\{u(\cdot): u(\cdot)~\mbox{is~a~weak~solution~on}~[T,\infty) \mbox{~with~} \omega\in\Omega\mbox{~satisfying~} \eqref{energine}-\eqref{energinese}, \\ \label{rantro1}
		& ~u(t)\in \mathbb{Y}(\omega) ,~\forall ~t\in [T,\infty)\big\},~\forall ~T\in\mathbb{R}.
\end{eqnarray}
\begin{eqnarray}\nonumber
		\mathcal{E}^\theta_\omega((\infty,\infty))=&\big\{u(\cdot): u(\cdot)~\mbox{is~a~weak~solution~on~}(-\infty,\infty) \mbox{~with~} \omega\in\Omega \mbox{~satisfying~}
			\eqref{energine}-\eqref{energinese}, \\ \label{rantro2}
		&u(t)\in \mathbb{Y}(\omega) ,~\forall ~t\in (-\infty,\infty)\big\}.
	\end{eqnarray}
Furthermore,  a map can be defined by
$$ R(t, \omega):P(H)\rightarrow P(H),$$
$$ R(t, \omega)A:=\{u(t):u(0)\in A, u\in\mathcal{E}^{\theta}_{\omega}([0,\infty))\},A\in P(H).$$
Using Lemma \ref{measurable} and Proposition  \ref{sem-con}, one obtains that  $R(\cdot, \cdot)\cdot$ is $\mathcal{B}(\mathbb{R}^+)\times\mathcal{F}\times\mathcal{B}(X)$ measurable. Therefore,
 $\{ R(t, \omega)\}=\{ R(t,\omega):t\geq0,t\in \mathbb{R}^+,  \omega\in\Omega\}$ is a random evolutionary semigroup.
 
From Theorem \ref{t3.3} and Lemma \ref{l5.1}, since there exists a family of random absorbing set and the random evolutionary semigroup $\{R(t,\omega)\}$ is weak upper semicontinuous, the existence of weak global attractor can be achieved immediately.
\begin{theorem}\label{weakattr}
Given  $f\in V^*$, $\eta_1\in L^{\frac{2}{1-p}}(D)$, $\eta_2\in L^{2}(D)$, then the weak global attractor $\{\mathcal{A}_w(\omega)\}$  exists and satisfies  $\mathcal{A}_w(\omega)=\Xi(0)$, where
$$\Xi(0)=\{u(0)|u\in \mathcal{E}^{\theta}_{\omega}((-\infty,\infty))\}.$$
\end{theorem}

In order to obtain the existence of strong global attractor,  strong asymptotic compactness of the random evolutionary semigroup must be achieved. Inspired by \cite{1, 2}, we use the following method to verify that $\{ R(t, \omega)\}$ is $s-$asymptotic compactness.
\begin{proposition}\label{precom}
 Suppose every complete trajectory $\mathcal{E}^{\theta}_{\omega}((-\infty,\infty))\in C((-\infty,\infty); H)$ and the stochastic evolutionary system $\mathcal{E}^{\theta}_{\omega}$ satisfies  the following three properties:
\begin{enumerate}
\item[A1] $\mathcal{E}^{\theta}_{\omega}([0,\infty))$ is a precompact set in $C([0,\infty); H_w)$;
\item[A2] Energy inequality: for all $ \epsilon>0$, there exists $\delta >0$ such that for every $u\in \mathcal{E}^{\theta}_{\omega}([0,\infty))$ and $t>0$,
$$
\|u(t)\|\leq \|u(t_0)\|+\epsilon, \mbox{~~~for~a.e.~} t_0\in (t-\delta, t);
$$
\item[A3] Strong convergence a.e.: if $u_k\in \mathcal{E}^{\theta}_{\omega}([0,\infty))$ is $d_{C([0,T]; H_w)}$-Cauchy sequence in $C([0,T]; H_w)$ for some $T>0$, then $\{u_k(t)\}$ is $d_s$-precompact sequence a.e. in $[0,T]$, where the strong distance $d_s(u,v)=\|u-v\|$ and weak distance $d_w$ is defined by \eqref{dwh}.
\end{enumerate}
Then $\{ R(t, \omega)\}$ is $s-$asymptotically compact.
\end{proposition}
\begin{proof}
Suppose the sequence $ u_n \in \mathcal{E}^{\theta}_{\omega}([0,\infty))$ satisfies $u_n \rightarrow u$ in $C([0,T]; H_w)$ as $n\rightarrow\infty$ for some $u\in \mathcal{E}^{\theta}_{\omega}([0,\infty))$. If $u(t)$ is strongly continuous at some $t = t^*\in (0,T)$, then $u_n(t^*) \rightarrow u(t^*)$ strongly in $H$.  In fact, by A3, there exists a set $I$ of measure zero and a subsequence (without relabel) satisfying that
$$
u_n(t) \rightarrow u(t) ~~\mbox{~strongly~in~} H \mbox{~on~} [0,T]\backslash I.
$$
Thanks to A2 and the strong continuity of $u$ at $t^*$, for every $\epsilon>0$, there exists $t_0\in [0, t^*)\backslash I$ satisfying
\begin{eqnarray}
\|u_n(t^*)\|\leq \|u_n(t_0)\|+\epsilon, \quad \|u(t_0)\|\leq \|u(t^*)\|+\epsilon, \mbox{~for~all~} n.
\end{eqnarray}
Then one has
\begin{eqnarray}
\|u(t^*)\|\leq \liminf\limits_{n\rightarrow \infty} \|u_n(t^*)\|\leq \limsup\limits_{n\rightarrow \infty} \|u_n(t^*)\|\leq \limsup\limits_{n\rightarrow \infty} \|u_n(t_0)\|+\epsilon\leq \|u(t^*)\|+2\epsilon,
\end{eqnarray}
which means that $u_n(t^*) \rightarrow u(t^*)$ strongly in $H$.

Next, let  $t_n\rightarrow\infty$ as $n\rightarrow\infty$ and  $x_n\in R(t_n,\theta_{-t_n}\omega)D$ for a bounded domain $D\subset H$, then  there exist $u_n\in \mathcal{E}^{\theta}_{\theta_{t_n} \omega}([-t_n,\infty))$ with $u_n(1)=x_n$.  By  A1 and a diagonalization process, there exits a subsequence $\{u_n\}$ (without relabel) and $u\in \mathcal{E}^{\theta}_{\omega}((-\infty,\infty))$ such that
$$
u_n \rightarrow u \mbox{~in~} C([0,2];   H_{w}).
$$
Since $u$ is continuous at $t^*=1$, then $x_n\rightarrow x$ in $H$ strongly, which finishes the proof.
\end{proof}

Next, we show that the random  evolutionary system $\mathcal{E}^\theta_\omega$ satisfies properties A1-A3.
\begin{lemma}\label{A1}
The random  evolutionary system $\mathcal{E}^\theta_\omega$ satisfies A1.
\end{lemma}
\begin{proof}
	  Let $\{u_{n}\}_{n=1}^\infty$ be a sequence of weak solutions to \eqref{NS-clor} with  $\omega\in\Omega$.  Similar to the proof Proposition \ref{sem-con},  one can  prove that there exists a subsequence  $\{u_{n}\}_{n=1}^\infty$  (without relabeling) converges to some $u$ in $C([t_1,t_2];  H_{w})$ for all $t_2>t_1$, i.e.
	\begin{eqnarray}\label{uniformcon}
		\langle u_n,\varphi \rangle\rightarrow \langle u, \varphi\rangle  \mbox{~uniformly~on~} [t_1,t_2], \mbox{~as~} n\rightarrow\infty,
	\end{eqnarray}
	for all $\varphi\in H$. From the definition of	weak solution, we know that $ \mathcal{E}^\theta_\omega([0,\infty))\subset C([0,\infty); H_{w}) $. Taking any sequence $\{u_{n}\}_{n=1}^{\infty} \in \mathcal{E}^\theta_\omega([0,\infty))$, by \eqref{uniformcon},  there exists a subsequence $\{u^{(1)}_{n}\}_{n=1}^\infty\subset\{u_{n}\}_{n=1}^{\infty}$  such that
	$$ (u^{(1)}_{n},\mathcal{B}^{(1)}_{n})\rightarrow (u^{(1)},\mathcal{B}^{(1)}) ~in~C([0,1];   H_{w})~~~~~as~n\to\infty, $$
	for some $ u^{(1)}\in  C([0,1];  H_{w}) $.
	Next, in the same process, there exists a subsequence $\{u^{(2)}_{n}\}_{n=1}^\infty\subset\{u^{(1)}_{n}\}_{n=1}^\infty$  such that
$$ u^{(2)}_{n}\rightarrow u^{(2)} ~in~C([0,2];  H_{w})~~~~~as~n\to\infty, $$
for some $ u^{(2)}\in  C([0,2]; H_{w}) $ with $  u^{(2)}=u^{(1)}$ on $ [0,1] $.

Last, repeating this diagonalization process, we can get a pair subsequence  $\{u^{(j+1)}_{n}\}_{n=1}^\infty\subset\{u^{(j)}_{n}\}_{n=1}^\infty$,  $j\in \mathbb{N}$,  satisfying that
$$ u^{(j+1)}_{n} \rightarrow u^{(j+1)} \mbox{~in~} C([0,j+1]; H_{w} )~~~~as~n\to \infty ,$$
for some $ u^{(j+1)}\in C([0,j+1];  H_{w})$ with $  u^{(j+1)}=u^{(j)}$ on $ [0,j] $.
Therefore, one obtains that a subsequence of $\{u_{n}\}_{n=1}^{\infty} \in \mathcal{E}^\theta_\omega([0,\infty))$  converges to some $u$ in $C([0,\infty);  H_{w})$,
which finishes the proof.
\end{proof}
\begin{lemma}\label{A2}
Given  $f\in V^*$, $\eta_1\in L^{\frac{2}{1-p}}(D)$, $\eta_2\in L^{2}(D)$, then the random  evolutionary system $\mathcal{E}^\theta_\omega$ satisfies A2, i.e.   for all $\varepsilon >0$, there exists $\delta>0$ such that
	$$ \|u(t)\|^2 \le\|u(t_{0})\|^2 + \varepsilon ~~~\mbox{~for~a.~e.~}t_{0} \in (t-\delta,t),$$
for every pair $ u \in\mathcal{E}([0,\infty))$ and $ t>0$.
\end{lemma}
\begin{proof}
	For all $\tilde{\varepsilon}>0$, there exists $\delta >0$ such that for all $t_{0} \in (t-\delta,t)$,
	\begin{eqnarray}\nonumber\label{3.14}
			&&\sup_{t\in\mathbb{R}}\int_{t_{0}}^{t}\big[\|f\|^2_{V^{*}}+\|\eta_1\|^{\frac{2}{1-p}}_{L^{\frac{2}{1-p}} }+\|\eta_2\|^2_{L^2}\big]ds\\
&\leq&\sup_{t\in\mathbb{R}} \int_{t-\delta}^{t}\big[\|f\|^2_{V^{*}}+\|\eta_1\|^{\frac{2}{1-p}}_{L^{\frac{2}{1-p}} }+\|\eta_2\|^2_{L^2}\big]ds<\tilde{\varepsilon}.
	\end{eqnarray}
Then taking $ u\in\mathcal{E}^\theta_\omega([0,\infty))$ and $t>0$, by Theorem \ref{exithe},	for a. e.  $0\leq t_0\leq t$,
	\begin{eqnarray}\nonumber
			\|u(t)\|^2 &\leq& e^ {-\frac{\nu\lambda_1}{2}t}\Big(e^ {\frac{\nu\lambda_1}{2}{t_0}}\|u(t_0)\|^2 +C \int_{t_0}^t e^ {\frac{\nu\lambda_1}{2}s}\big(\|f\|^2_{V^\ast}+  |\zeta_{\delta}(\theta_s\omega)|^{\frac{2}{1-p}}\|\eta_1\|^{\frac{2}{1-p}}_{L^{\frac{2}{1-p}} }\\
&&+  |\zeta_{\delta}(\theta_s\omega)|^{2}\|\eta_2\|^2_{L^2} \big) ds  \Big),
	\end{eqnarray}
Then by combining \eqref{omeest} and \eqref{3.14},  one has
	\begin{equation*}
			\|u(t)\|^2\le \|u(t_0)\|^{2}+C\int_{t_{0}}^{t}\big[\|f\|^2_{V^{*}}+\|\eta_1\|^{\frac{2}{1-p}}_{L^{\frac{2}{1-p}} }+\|\eta_2\|^2_{L^2}\big]ds\leq\|u_{0}\|^{2}+\varepsilon,
	\end{equation*}
	for a.~e. $t_{0} \in (t-\delta,t)$, which finishes the proof.
\end{proof}

\begin{lemma}\label{A3}
Given  $f\in V^*$, $\eta_1\in L^{\frac{2}{1-p}}(D)$, $\eta_2\in L^{2}(D)$, then the random  evolutionary system $\mathcal{E}^\theta_\omega$ satisfies A3, i.e. if $ \{u_{n}\}_{n=1}^\infty,~u \in \mathcal{E}^\theta_\omega([0,\infty))$ satisfying that  $u_{n}  \to u $  in $C([0,T];H_{w})$ for some $ T>0 $ as $n\rightarrow\infty$. Then there exists a subsequence $ u_{n_k}(t) \to  u(t) $ in $  H $ a. e. on  $ [0,T] $ as $n_k\rightarrow \infty$.
\end{lemma}
\begin{proof}
	Taking a sequence $u_{n} \subset \mathcal{E}^\theta_\omega ([0,\infty))$ satisfying $u_{n}\to u$ in $ C([0,T]; H_{w} ) $ as $ n\to\infty $ for some $u\in C([0,T]; H_{w})$.
	Since $u_{n}$ is a pair of weak solutions of equation \eqref{NS-clor},
similar with  \eqref{comconver1} in Theorem \ref{exithe}, by Aubin-Lions lemma,  one can choose a subsequence
	$$ \int_{0}^{T} \|u_{n_k}(s)-u(s)\|^2 ds \to 0 ~~~~\mbox{as}~n_k\to \infty.$$
	Specially, one has
	\begin{equation}\label{4.33}
		\|u_{n_k} (t)-u(t)\|\rightarrow 0~ ~ \mbox{as}~  n_k\to\infty ~\mbox{a. e. on}~ [0,T],
	\end{equation}
	which finishes the proof.
\end{proof}

\begin{theorem}\label{th2.4}
For given $f, \eta_1, \eta_2$ are in $V^{*}$, $ L^{\frac{2}{1-p}}(D) $ and $  L^{2}(D) $ respectively, suppose every complete trajectory $\mathcal{E}^{\theta}_{\omega}((-\infty,\infty))\in C((-\infty,\infty); H)$.   Then the strong global attractor $\{\mathcal{A}_s(\omega)\}$ exists and $\{\mathcal{A}_s(\omega)\}=\{\mathcal{A}_w(\omega)\}, \omega\in \Omega$.
\end{theorem}
\begin{proof}
 Since the evolutionary system $\mathcal{E}^\theta_\omega$ satisfies the results of Lemma \ref{A1}-Lemma \ref{A3}, from Proposition \ref{precom}, one knows that   $\mathcal{E}^\theta_\omega$ is $s-$asymptotically compact. Finally, by using  Theorem \ref{t4.1} and combining Lemma \ref{attrset1},  Proposition  \ref{sem-con},  the  main result can be achieved.
\end{proof}

\begin{remark}
It is interesting use our established results in this paper to other models like 3D Euler equation (\cite{15}), 3D Navier-Stokes equations with space-time white noise (\cite{19}) or to analyze other types solution of 3D Navier-Stokes equations (\cite{16,17,18}).
\end{remark}

\subsection{Tracking property}
In this subsection, the weak and strong tracking properties of random evolutionary system  is showed.
\begin{theorem}(Weak Tracking Property)
Let $\mathcal{E}^{\theta}_{\omega}([T,\infty)),T\in \mathbb{R}^+$ be a random evolutionary system constructed by \eqref{rantro1}-\eqref{rantro2}. Let $A\subset X$. Then for any $\varepsilon>0$, there exists $t_0\geq0$, such that for any $t^*>t_0$, every random trajectory $u\in\mathcal{E}^{\theta}_{\omega}([t^*,\infty))$ with $u(t^*)\in A$ satisfies
$$d_{C([t^*,\infty);X_w)}(u,v)<\varepsilon,$$
for some complete trajectory $v\in \mathcal{E}^{\theta}_{\omega}((-\infty,\infty))$ with $v(t)\in \Omega_w(A,\omega)$ for all $t\in \mathbb{R}, \omega\in\Omega$.
\end{theorem}
\begin{proof}
By Lemma \ref{A1}, $\mathcal{E}^{\theta}_{\omega}([T,\infty))$ is a compact set in $C([T,\infty);X_w)$. Then from Theorem \ref{t5.3}, the weak tracking property is obtained.
\end{proof}

\begin{theorem}(Strong Tracking Property)
Let $\mathcal{E}^{\theta}_{\omega}([T,\infty)),T\in \mathbb{R}^+$ be a random evolutionary system constructed by \eqref{rantro1}-\eqref{rantro2}. Suppose every complete trajectory $\mathcal{E}^{\theta}_{\omega}((-\infty,\infty))\in C((-\infty,\infty); H)$. 
   Let $A\subset X$. Then for any $\varepsilon>0$ and $T>0$, there exists $t_0\geq0$, such that for any $t^*>t_0$, every trajectory $u\in \mathcal{E}^{\theta}_{\omega}([t^*,\infty))$ with $u(t)\in A$ satisfies
$$d_s(u(t),v(t))<\varepsilon,\forall t\in[t^*,t^*+T],$$
for some complete trajectory $\mathcal{E}^{\theta}_{\omega}((-\infty,\infty))$ with $v(t)\in \Omega_s(A,\omega)$ for all $t\in \mathbb{R}$.
\end{theorem}
\begin{proof}
From the proof of Theorem \ref{th2.4}, one knows that $\mathcal{E}^\theta_\omega$ is $s-$asymptotically compact.
By Theorem \ref{sptp},   the strong tracking property is obtained.
\end{proof}
\bigskip

{\bf Acknowledgments}
The first author would like to thank Prof.~Songsong Lu, department of mathematics, sun yat-sen university, guangzhou, 510275,
PR China, for helpful discussions.


\end{document}